\documentclass{amsart}
\usepackage{amssymb, latexsym}
\usepackage{url}

\DeclareMathOperator{\sgn}{\mathrm{sgn}}

\begin{document}
 \bibliographystyle{plain}

 \newtheorem{theorem}{Theorem}
 \newtheorem{lemma}{Lemma}
 \newtheorem{corollary}{Corollary}
 \newtheorem{conjecture}{Conjecture}
 \newtheorem{definition}{Definition}
 \newcommand{\mc}{\mathcal}
 \newcommand{\zpz}{\mathbb{Z}/p\mathbb{Z}}
 \newcommand{\mbb}{\mathbb}
 \newcommand{\A}{\mc{A}}
 \newcommand{\B}{\mc{B}}
 \newcommand{\cc}{\mc{C}}
 \newcommand{\D}{\mc{D}}
 \newcommand{\E}{\mc{E}}
 \newcommand{\F}{\mc{F}}
 \newcommand{\I}{\mc{I}}
 \newcommand{\J}{\mc{J}}
 \newcommand{\K}{\mc{K}}
 \newcommand{\M}{\mc{M}}
 \newcommand{\nn}{\mc{N}}
 \newcommand{\qq}{\mc{Q}}
 \newcommand{\U}{\mc{U}}
 \newcommand{\X}{\mc{X}}
 \newcommand{\Y}{\mc{Y}}
 \newcommand{\itQ}{\mc{Q}}
 \newcommand{\C}{\mathbb{C}}
 \newcommand{\R}{\mathbb{R}}
 \newcommand{\N}{\mathbb{N}}
 \newcommand{\Q}{\mathbb{Q}}
 \newcommand{\T}{\mathbb{T}}
 \newcommand{\oU}{\overline U}
 \newcommand{\Z}{\mathbb{Z}}
 \newcommand{\ff}{\mathfrak F}
 \newcommand{\fb}{f_{\beta}}
 \newcommand{\fg}{f_{\gamma}}
 \newcommand{\gb}{g_{\beta}}
 \newcommand{\h}{\frac12}
 \newcommand{\hh}{\tfrac12}
 \newcommand{\balpha}{\boldsymbol \alpha}
 \newcommand{\bb}{\boldsymbol b}
 \newcommand{\bbbeta}{\boldsymbol \beta}
 \newcommand{\be}{\boldsymbol e}
 \newcommand{\bg}{\boldsymbol g}
 \newcommand{\bj}{\boldsymbol j}
 \newcommand{\bk}{\boldsymbol k}
 \newcommand{\bm}{\boldsymbol m}
 \newcommand{\bgamma}{\boldsymbol \gamma}
 \newcommand{\bo}{\boldsymbol 0}
 \newcommand{\bt}{\boldsymbol t}
 \newcommand{\bu}{\boldsymbol u}
 \newcommand{\bv}{\boldsymbol v}
 \newcommand{\bx}{\boldsymbol x}
 \newcommand{\bwy}{\boldsymbol y}
 \newcommand{\bxi}{\boldsymbol \xi}
 \newcommand{\bbeta}{\boldsymbol \eta}
 \newcommand{\bw}{\boldsymbol w}
 \newcommand{\bz}{\boldsymbol z}
 \newcommand{\vphi}{\varphi}
 \newcommand{\dx}{\text{\rm d}\bx}
 \newcommand{\dy}{\text{\rm d}\bwy}
 \newcommand{\dmu}{\text{\rm d}\mu}
 \newcommand{\dnu}{\text{\rm d}\nu}
 \def\today{\number\time, \ifcase\month\or
  January\or February\or March\or April\or May\or June\or
  July\or August\or September\or October\or November\or December\fi
  \space\number\day, \number\year}
  
\title[group actions]{A Dirichlet approximation theorem\\for group actions}
\author{Clayton~Petsche}
\author{Jeffrey~D.~Vaaler}
\thanks{This research was supported by NSA grant, H98230-12-1-0254.}

\address{Department of Mathematics, Oregon State University, Corvallis, Oregon 97331 USA}
\email{petschec@math.oregonstate.edu}

\address{Department of Mathematics, University of Texas, Austin, Texas 78712 USA}
\email{vaaler@math.utexas.edu}
\numberwithin{equation}{section}

\begin{abstract}  If $G$ is a compact group acting continuously on a compact metric space $(X, m)$, we prove two
results that generalize Dirichlet's classical theorem on Diophantine approximation.  If $G$ is a noncommutative 
compact group of isometries, we obtain a noncommutative form of Dirichlet's theorem.  We apply our  general result to the 
special case of the unitary group $U(N)$ acting on the complex unit sphere, and obtain a noncommutative result in this setting. 
\end{abstract}

\maketitle

\section{Introduction}

Let $G$ be a compact topological group that acts continuously and faithfully on a compact metric space $\{X,m\}$.  If 
$\A \subseteq G$ is a finite subset containing at least two points, we prove two results that establish the existence of a 
nonidentity element in the difference set 
\begin{equation*}\label{intro1}
\big\{a b^{-1} : a \in \A,\ b\in \A,\ \text{and $a \not= b$}\big\}
\end{equation*}
that moves the points of $X$ minimally.  In the special case of the group $G = (\R/\Z)^N$ acting on $X = (\R/\Z)^N$ by 
translation, our results reproduce Dirichlet's classical theorem on Diophantine approximation.  If $G$ acts 
continuously as a group of isometries on $\{X, m\}$, we obtain a noncommutative form of Dirichlet's theorem.   

Before developing the general setting in which we formulate our results, we consider the special case of the unitary
group $G = U(N)$ of $N \times N$ unitary matrices.  The group $U(N)$ acts continuously and faithfully on the set
\begin{equation}\label{intro5}
X = \{\bx \in \C^N : |\bx|_2 = 1\},
\end{equation}  
where
\begin{equation*}\label{intro9}
|\bx|_2 = \bigl(|x_1|^2 + |x_2|^2 + \cdots + |x_N|^2\bigl)^{\h}
\end{equation*}
is the standard Hermitian norm on (column) vectors $\bx$ in $\C^N$.  If $\bx$ and $\bwy$ are points of $X$, we use the
metric $m$ given by
\begin{equation*}\label{intro13}
m(\bx, \bwy) = |\bx - \bwy|_2.
\end{equation*}
Then a unitary matrix $A$ in $U(N)$ acts on vectors $\bx$ in $X$ by $(A, \bx) \mapsto A\bx$, and it is easy to verify that this action
is an isometry.  We define a function $\vphi : U(N) \rightarrow [0, \infty)$ by
\begin{equation}\label{intro17}
\vphi(A) = \sup\big\{|A\bx - \bx|_2 : \bx \in X\big\},
\end{equation}
so that $\vphi(A)$ measures the maximum distance that $A$ moves a point $\bx$ in $X$.  It is instructive to observe that the
map $(A, B) \mapsto \vphi\bigl(A B^{-1}\bigr)$ defines a metric on the compact group $U(N)$, and this metric induces its group topology.

\begin{theorem}\label{thmintro1}  Let $\A \subseteq U(N)$ be a finite subset of cardinality $|\A| \ge 2$.  If
\begin{equation}\label{intro23}
\delta(\A) = \min\big\{\vphi\bigl(A B^{-1}\bigr): A\in\A,\ B\in\A,\ \text{and $A\not= B$}\big\},
\end{equation} 
then we have
\begin{equation}\label{intro33}
\delta(\A) \le 2 \pi |\A|^{-1/N^2}.
\end{equation}
\end{theorem} 

As an application of Theorem \ref{thmintro1}, we obtain the following noncommutative form of Dirichlet's theorem.

\begin{theorem}\label{thmintro2}  Let $A$ and $B$ be matrices in the unitary group $U(N)$, and let $J$ and $K$ be positive integers.  If
\begin{equation}\label{intro43}
\delta_{J, K}(A, B) = \min\big\{\vphi\bigl(A^j B^k\bigr) : \text{$|j| \le J$, $|k| \le K$, and $(j, k) \not= (0, 0)$}\big\},
\end{equation}
then
\begin{equation}\label{intro49}
\delta_{J, K}(A, B) \le 2\pi (J + 1)^{-1/N^2} (K + 1)^{-1/N^2}.
\end{equation}
\end{theorem}

If $A$, $B$, and $C$, are three matrices in $U(N)$, and $J$, $K$, and $L$ are positive integers, we can form the nonnegative number
\begin{align*}\label{intro53}
\begin{split}
\delta_{J, K, L}(A, B, C) &= \min\big\{\vphi\bigl(A^j B^k C^{\ell}\bigr) : \text{$|j| \le J$, $|k| \le K$, $|\ell| \le L$,}\\ 
				      &\qquad\qquad\qquad\qquad \text{and $(j, k, \ell) \not= (0, 0, 0)$}\big\}.
\end{split}
\end{align*}
It would be of interest to give an upper bound for this quantity that is analogous to (\ref{intro49}).   But such a generalization 
is not obvious using the methods developed here.

In section 2 we describe the setting in which we prove our general results.  This will be
familiar to researchers on continuous group actions.  We include it mainly for convenience and to establish notation.
In section 3 we prove our general form of Dirichlet's theorem stated as Theorem \ref{thm3}.  And in 
section 4 we prove a noncommutative version of this result which we state as Theorem \ref{thm4}.
In section 5 we show that Theorem \ref{thm3} does in fact reproduce Dirichlet's classical 
theorem in Diophantine approximation when $G = (\R/\Z)^N$ acts on $X = (\R/\Z)^N$ by 
translation.  The final section contains proofs of Theorem \ref{thmintro1} and Theorem \ref{thmintro2} for the unitary group.


\section{Continuous actions of compact groups}

Let $G$ be a compact topological group and write $e$ for the identity element in $G$.  Let $\{X, m\}$ be a compact metric space, and assume that 
$G\times X\rightarrow X$ is a faithful, continuous action of $G$ on $X$, which we denote by $(g, x)\mapsto gx$.  
As the maps $(g,x)\mapsto gx$ and $(g,x)\mapsto x$ are both continuous, it follows that $(g,x)\mapsto (gx,x)$ is a continuous map from $G\times X$ into 
$X\times X$.  Then 
\begin{equation*}\label{group31}
(g,x)\mapsto m(gx,x)
\end{equation*} 
is a continuous map from $G\times X$ into $[0,\infty)$.  Since $X$ is compact,
\begin{equation}\label{group35}
\vphi(g) = \sup\{m(gx,x): x\in X\}
\end{equation}
is finite, and therefore (\ref{group35}) defines a map $\vphi:G\rightarrow [0,\infty)$.  The map $\vphi$ will play a basic role in 
this work.

We derive some elementary properties of the map $\vphi$.  If $g$ and $h$ are in $G$, we have
\begin{equation}\label{group41}
\vphi(g^{-1}) = \sup\{m(g^{-1}x,x): x\in X\} = \sup\{m(x,gx): x\in X\} = \vphi(g),
\end{equation}
and 
\begin{equation}\label{group45}
\begin{split}
\vphi(gh^{-1}) &= \sup\{m(gh^{-1}x,x): x\in X\}\\
	&= \sup\{m(gx,hx): x\in X\}\\
	&\le \sup\{m(gx,x) + m(x,hx): x\in X\}\\
	&\le \sup\{m(gx,x): x\in X\} + \sup\{m(y,hy): y\in X\}\\
	&= \vphi(g) + \vphi(h). 
\end{split}\end{equation}
If the metric $m$ is nonarchimedean, that is, if $m$ satisfies the strong triangle inequality
\begin{equation*}\label{group51}
m(x,y) \le \max\{m(x,z), m(z,y)\}
\end{equation*}
for all $x$, $y$, and $z$ in $X$, then we get the corresponding nonarchimedean inequality
\begin{equation}\label{group55}
\vphi(gh^{-1}) \le \max\{\vphi(g), \vphi(h)\}.
\end{equation}

Using (\ref{group41}) and (\ref{group45}) we find that
\begin{equation*}\label{group61}
\vphi(g) = \vphi(gh^{-1}h) \le \vphi(gh^{-1}) + \vphi(h), 
\end{equation*}
and similarly
\begin{equation*}\label{group65}
\vphi(h) = \vphi(hg^{-1}g) \le \vphi(hg^{-1}) + \vphi(g) = \vphi(gh^{-1}) + \vphi(g). 
\end{equation*}
Hence we get
\begin{equation}\label{group71}
\bigl|\vphi(g) - \vphi(h)\bigr| \le \vphi(gh^{-1}).
\end{equation}

Let $\rho:G\times G\rightarrow [0,\infty)$ be defined by $\rho(g,h) = \vphi(gh^{-1})$.  Because $G$ acts faithfully, $\rho(g,h) = 0$ if
and only if $g = h$.  The identity (\ref{group41}) shows that $\rho(g,h) = \rho(h,g)$.  If $g$, $h$ and $k$ are elements of $G$ then
using (\ref{group45}) we get 
\begin{equation*}\label{group101}
\rho(g,h) = \vphi\bigl(gk^{-1}(hk^{-1})^{-1}\bigr) \le \vphi(gk^{-1}) + \vphi(hk^{-1}) = \rho(g,k) + \rho(k,h).
\end{equation*}
This shows that $\rho$ is a metric on $G$ that is induced by the metric $m$ on $X$.  If $m$ is nonarchimedean, then it follows 
using (\ref{group55}) that $\rho$ is also nonarchimedean.  

Of course, the metric $\rho$ does not necessarily induce the group topology in $G$.  We would like to work in a situation where $\rho$ induces a metric topology in which every $\rho$-open 
set is also open in the group topology of $G$.  That is, $\rho$ induces a topology in $G$ that is weaker than the group topology.  In order for this to happen it is necessary and sufficient that $\vphi$ be a continuous map.  This follows from the following general principle.

\begin{lemma}\label{lem3}
Assume that $Y$ is a locally compact Hausdorff space, $Z$ is a compact Hausdorff space, and $f:Y\times Z\rightarrow \R$ is a 
continuous map.  Then the function $g:Y\rightarrow \R$ defined by
\begin{equation*}\label{group171}
g(y) = \sup\{f(y,z): z\in Z\}
\end{equation*}
is continuous.
\end{lemma}

\begin{proof}  Let $y$ be a point in $Y$.  By local compactness there exists an open neighborhood $U$ of $y$ and a compact set $K$
such that $U\subseteq K \subseteq Y$.  Then it suffices to show that the restriction of $g$ to $K$ is continuous.

Let $\epsilon > 0$ and for each point $(k, z)$ in $K\times Z$ let
\begin{equation*}\label{group175}
B(k, z) = \{(\alpha, \beta)\in K\times Z: |f(k, z) - f(\alpha, \beta)| < \epsilon\}.
\end{equation*}
Then $B(k, z)$ is an open neighborhood of $(k, z)$.  Let $U(k)$ be an open neighborhood of $k$ and let $V(z)$ be an open neighborhood
of $z$ such that $U(k)\times V(z)\subseteq B(k, z)$.  The collection of open sets
\begin{equation*}\label{group181}
\big\{U(k)\times V(z): (k, z)\in K\times Z\big\}
\end{equation*}
covers the compact space $K\times Z$, and so there exists a finite subcover
\begin{equation*}\label{group185}
\big\{U(k_n)\times V(z_n): (k_n, z_n)\in K\times Z\ \text{and}\ n = 1, 2, \dots , N\big\}.
\end{equation*}
Define a continuous function $h:K\rightarrow \R$ by
\begin{equation*}\label{group191}
h(k) = \max\{f(k, z_n): n = 1, 2, \dots , N\}.
\end{equation*}
Since $Z$ is compact, at each point $k$ in $K$ there exists a point $\zeta = \zeta(k)$ in $Z$ such that $g(k) = f(k, \zeta)$.  Then
there exists an integer $m$ with $1 \le m \le N$ such that $(k, \zeta)$ belongs to $U(k_m)\times V(z_m)$.  It follows that
\begin{align*}\label{group195}
\begin{split}
0 &\le g(k) - h(k)\\
  &\le f(k, \zeta) - f(k, z_m)\\
  &\le |f(k, \zeta) - f(k_m, z_m)| + |f(k_m, z_m) - f(k, z_m)|\\
  & < 2\epsilon.
\end{split}  
\end{align*}
Since $h$ is continuous and $\epsilon > 0$ is arbitrary, it follows that $g$ restricted to $K$ is continuous.
\end{proof}

\begin{corollary}\label{cor2}
The map $\vphi:G\rightarrow [0,\infty)$ defined by {\rm (\ref{group35})} is continuous.
\end{corollary}

\begin{proof}
In Lemma \ref{lem3}, take the map $f:G\times X\to [0, \infty)$ given by $f(g,x)=m(gx, x)$.
\end{proof}

Our assumption that the $G$-action on $X$ is faithful is not restrictive.  Indeed, if the action were not faithful, then we could replace 
$G$ by $G/K$, where 
\begin{equation*}\label{group91}
K = \{g\in G: \vphi(g) = 0\},
\end{equation*}
which is easily seen to be a closed, normal subgroup of $G$.  Indeed, $k$ is in $K$ if and only if $kx = x$ for all $x$ in $X$.  If $g$ is in $G$ and $k$ is in $K$ then
\begin{equation*}\label{group95}
m(gkg^{-1}x, x) = m(gg^{-1}x, x) = m(x, x) = 0,
\end{equation*}
and therefore $K$ is normal in $G$.  If $gh^{-1}$ belongs to $K$ then (\ref{group71}) implies that
$\vphi(g) = \vphi(h)$.  In particular, $\vphi$ is constant on each coset $hK$.  Thus $\vphi$ induces a map $\overline{\vphi}:G/K\rightarrow [0,\infty)$ by defining $\overline{\vphi}(hK) = \vphi(h)$, and $G/K$ acts faithfully on $X$.


\section{A generalization of Dirichlet's theorem}

Let $\mu$ denote a Haar measure on the Borel subsets of $G$ normalized  so that $\mu(G) = 1$.  We recall that for a compact group the Haar measure $\mu$ is both left and right invariant.  
That is, if $E\subseteq G$ is a Borel set and $g$ is in $G$, then $\mu(E) = \mu(gE) = \mu(Eg)$.  For $0 \le t$ we define the
distribution function
\begin{equation}\label{group245}
\Phi(t) = \mu\{g\in G: \vphi(g) < t\}.
\end{equation}
Clearly $\Phi(0) = 0$, and $t\mapsto \Phi(t)$ is nondecreasing.  If $0 < t$ then 
\begin{equation*}\label{group251}
\{g\in G: \vphi(g) < t\}
\end{equation*}
is a nonempty open set, and therefore it has positive Haar measure.  If $\lim_{t\rightarrow 0+} \vphi(t) > 0$ then it follows 
that $\{e\}$ is an open set and $G$ has the discrete topology.  As $G$ is assumed to be compact, we find that $G$ is finite in this case.

Next we prove the following result which is inspired by Blichfeldt's theorem \cite{blichfeldt1914} in the geometry of numbers (see also
\cite[Chapter III.2, Theorem I]{cassels1971}.)

\begin{theorem}\label{thm3}
Let $\A\subseteq G$ be a finite, nonempty set.  Write $|\A|$ for the cardinality of $\A$, and define
\begin{equation}\label{group261}
\delta(\A) = \min\{\vphi(ab^{-1}): a\in\A,\ b\in\A,\ \text{and $a\not= b$}\}.
\end{equation}
Then we have
\begin{equation}\label{group265}
\Phi\bigl(\tfrac12 \delta(\A)\bigr) \le |\A|^{-1}.
\end{equation}
If the metric $m$ is nonarchimedean then
\begin{equation}\label{group271}
\Phi\bigl(\delta(\A)\bigr) \le |\A|^{-1}.
\end{equation}
\end{theorem}

\begin{proof}
As $\vphi(g) = 0$ if and only if $g = e$, we may assume that $\delta(\A)$ is positive.
Let $V\subseteq G$ be the nonempty open set
\begin{equation*}\label{group275}
V = \big\{g\in G: \vphi(g) < \tfrac12 \delta(\A)\big\}.
\end{equation*} 
For each point $a$ in $\A$ we have
\begin{equation}\label{group281}
Va = \big\{ga\in G: \vphi(g) < \tfrac12 \delta(\A)\big\} = \big\{g\in G: \vphi(ga^{-1}) < \tfrac12 \delta(\A)\big\}.
\end{equation}
Assume that $a$ and $b$ are distinct points in $\A$ such that $Va\cap Vb$ contains a point $h$.  In view of (\ref{group281}) we have 
\begin{equation*}\label{group285}
\vphi(ha^{-1}) < \tfrac12 \delta(\A)\quad\text{and}\quad \vphi(hb^{-1}) < \tfrac12 \delta(\A).
\end{equation*}
Using (\ref{group41}) and (\ref{group45}) we also get
\begin{align*}\label{group291}
\begin{split}
\vphi(ab^{-1}) &= \vphi(ah^{-1}hb^{-1})\\
	                &\le \vphi(ha^{-1}) + \vphi(hb^{-1})\\
	                &< \tfrac12 \delta(\A) + \tfrac12 \delta(\A),
\end{split}	
\end{align*}
which is impossible.  It follows that the subsets $Va$ with $a\in \A$ are disjoint.  Using the right translation invariance 
of Haar measure we find that
\begin{equation*}\label{group295}
\sum_{a\in \A} \mu(Va) = |\A| \mu(V) = |\A| \Phi\bigl(\tfrac12 \delta(\A)\bigr) \le 1,
\end{equation*}
and this verifies (\ref{group265}).  

If $m$ is nonarchimedean then $\vphi$ also satisfies the nonarchimedean inequality (\ref{group55}).  Arguing as before, we are led to 
the inequality (\ref{group271}).
\end{proof}

\begin{corollary}
Let $a$ be an element of the compact group $G$ and define
\begin{equation}\label{group305}
\delta_N(a) = \min\{\vphi(a^n): 1\le n\le N\}.
\end{equation}
Then we have
\begin{equation}\label{group311}
\Phi\bigl(\tfrac12 \delta(a)\bigr) \le (N+1)^{-1}.
\end{equation}
If $m$ is nonarchimedean then
\begin{equation}\label{group315}
\Phi\bigl(\delta(a)\bigr) \le (N+1)^{-1}.
\end{equation}
\end{corollary}

\begin{proof}
Apply the Theorem with $\A = \{a^n: n=0, 1, 2, \dots , N\}$, and use the fact that $\vphi(a^{-n}) = \vphi(a^n)$.
\end{proof}

Given an element $a\in G$ of infinite order, we could also consider the closure $H$ of $\langle a\rangle$ in $G$, which is a compact subgroup (\cite{stroppel2006}, 6.26).  Then we can apply the above corollary to $H$.  Of course, the normalized Haar measure on $H$ is not necessarily the restriction to $H$ of 
a Haar measure on $G$. 


\section{$G$ is a group of isometries}

Using the metric $m$ on $X$, we may define the closed subgroup
\begin{equation}\label{group11}
I = \{g\in G: m(gx, gy) = m(x, y)\ \text{for all}\ (x, y)\in X\times X\}
\end{equation}
of all isometries in $G$.  If $g$ and $h$ are elements of $G$ which do not commute, then in general we do not expect to have 
$\vphi(gh) = \vphi(hg)$.  However, if either $g$ or $h$ is an element of the closed subgroup $I$ of all isometries in $G$, then we do have 
\begin{equation}\label{group75}
\vphi(gh) = \vphi(hg).
\end{equation}
For example, if $g$ is an isometry then
\begin{align}\label{group81}
\begin{split}
\vphi(gh) &= \sup\{m(ghx, x): x\in X\}\\
	&= \sup\{m(ghx,gg^{-1}x): x\in X\}\\
	&= \sup\{m(hx,g^{-1}x): x\in X\}\\
	&= \sup\{m(hgx,x): x\in X\}\\
	&= \vphi(hg).
\end{split}
\end{align}
Alternatively, if $g$ is an isometry then
\begin{equation}\label{group85}
\vphi(g h g^{-1}) = \vphi(h)
\end{equation} 
for all $h$ in $G$.  In particular, if $G = I$ is a group of isometries, then $\vphi : G \rightarrow [0, \infty)$ is constant on conjugacy classes.  In this case we may derive the following noncommutative form of Dirichlet's theorem on Diophantine approximation.

\begin{theorem}\label{thm4}
Assume that $G = I$, and let $a$ and $b$ belong to $G$.  For positive integers $M$ and $N$ define
\begin{equation}\label{group321} 
\delta_{M,N}(a,b) = \min\big\{\vphi\bigl(a^mb^n\bigr): |m|\le M, |n|\le N,\ \text{and}\ (m,n)\not= (0,0)\big\}.
\end{equation}
Then we have
\begin{equation}\label{group325}
\Phi\bigl(\tfrac12 \delta_{M,N}(a,b)\bigr) \le (M+1)^{-1}(N+1)^{-1}.
\end{equation}
If the metric $m$ is nonarchimedean, then
\begin{equation}\label{group331}
\Phi\bigl(\delta_{M,N}(a,b)\bigr) \le (M+1)^{-1}(N+1)^{-1}.
\end{equation}
\end{theorem}

\begin{proof}
If there exist integers $k$ and $l$, not both zero, such that $|k| \le M$, $|l| \le N$ and $a^kb^l = e$, then we have
$\delta_{M,N}(a,b) = 0$ and the result is obvious.  Therefore we define
\begin{equation*}\label{group335}
\A = \{a^{k}b^{l}: 0\le k\le M\ \text{and}\ 0\le l\le N\},
\end{equation*}
and we assume that $|\A| = (M+1)(N+1)$.  If $a^{k_1}b^{l_1}$ and $a^{k_2}b^{l_2}$ are distinct elements of $\A$ then, making use of
(\ref{group75}), we have
\begin{align*}\label{group341}
\begin{split}
\vphi\bigl(a^{k_1}b^{l_1}\bigl(a^{k_2}b^{l_2}\bigr)^{-1}\bigr) &= \vphi\bigl(a^{k_1}b^{l_1-l_2}a^{-k_2}\bigr)\\
	&= \vphi\bigl(a^{k_1-k_2}b^{l_1-l_2}\bigr). 
\end{split}
\end{align*}
This shows that $\delta_{M,N}(a,b)$ as defined by (\ref{group321}) is equal to $\delta(\A)$ as defined by (\ref{group261}).  Hence
the inequalities (\ref{group325}) and (\ref{group331}) both follow from Theorem \ref{thm3}.
\end{proof}

\section{The classical case of Dirichlet's Theorem}

Let $\|\ \|:\R\rightarrow [0, \h]$ be defined by
\begin{equation*}\label{class1}
\|x\| = \min\{|x - n|: n\in \Z\},
\end{equation*}
so that $\|x\|$ is the distance from the real number $x$ to the nearest integer.  Clearly $x\mapsto \|x\|$ is constant on each
coset of the quotient group $\R/\Z$, and so we have $\|\ \|:\R/\Z\rightarrow [0, \h]$.  Then $(x, y)\mapsto \|x - y\|$ is a 
metric on $\R/\Z$ which induces its quotient topology.  More generally,
\begin{equation}\label{class2}
(\bx, \bwy) \mapsto \max\{\|x_l - y_l\|: 1 \le l \le L\}
\end{equation}
is a metric on the product $(\R/\Z)^L$ which induces its product topology.

Now suppose that $G = (\R/\Z)^L$.  Let $X = (\R/\Z)^L$ with metric defined by (\ref{class2}), and let $G$ act on $X$ by translation.  
Here the group is naturally written using additive notation, so that the action is given by $(\bg, \bx) \mapsto \bg + \bx$.  It is
trivial that this action is continuous.  In this case we find that
\begin{align}\label{class3}
\begin{split}
\vphi(\bg) &= \sup\big\{\max\{\|g_l + x_l - x_l\|; 1 \le l \le L\}: \bx \in X\big\}\\
           &= \max\{\|g_l\|: 1 \le l \le L\}.
\end{split}
\end{align}
And for $0 \le t \le \h$ we get
\begin{equation}\label{class4}
\Phi(t) = \mu\{\bg \in G: \vphi(\bg) < t\} = (2t)^L.
\end{equation}

Let $\balpha_1, \balpha_2, \dots , \balpha_M$ be nonzero points in $G$, and let $K_1, K_2, \dots , K_M$ be positive integers.
Then define
\begin{equation*}\label{class5}
\A = \big\{k_1\balpha_1 + k_2\balpha_2 + \cdots + k_M\balpha_M: 0 \le k_m \le K_m\big\}.
\end{equation*}
If
\begin{equation*}\label{class6}
|\A| < \prod_{m=1}^M (K_m + 1),
\end{equation*}
then it is trivial that there exists a nonzero integer vector $\bj$ in $\Z^M$ such that
\begin{equation*}\label{class7}
j_1\balpha_1 + j_2\balpha_2 + \cdots + j_M\balpha_M = \bo,
\end{equation*}
and $|j_m| \le K_m$ for each $m = 1, 2, \dots , M$.  Thus we assume that
\begin{equation*}\label{class8}
|\A| = \prod_{m=1}^M (K_m + 1).
\end{equation*}
As in the statement of Theorem \ref{thm3}, let
\begin{equation*}\label{class9}
\delta(\A) = \min\big\{\vphi(j_1\balpha_1 + j_2\balpha_2 + \cdots + j_M\balpha_M): \bj \not= \bo,\ \text{and}\ |j_m| \le K_m + 1\big\}.
\end{equation*}
Then the inequality (\ref{group265}) and (\ref{class4}) imply that
\begin{equation}\label{class10}
\delta(\A)^L \le \prod_{m=1}^M (K_m + 1)^{-1}.
\end{equation}
This is a slight generalization of Theorem VI in \cite[Chapter 1]{cassels1965}, and includes Dirichlet's
basic theorem on Diophantine approximation. 

\section{Proof of Theorem \ref{thmintro1} and Theorem \ref{thmintro2}}

Throughout this section we write $\mu$ for Haar measure on the Borel subsets of the unitary group $U(N)$ normalized so 
that $\mu\bigl(U(N)\bigr) = 1$.  We let $X$ denote the surface (\ref{intro5}) of the complex unit ball, and let
\begin{equation*}\label{one1}
\vphi : U(N) \rightarrow [0, 2]
\end{equation*}
denote the function defined by (\ref{intro17}).  We write
\begin{equation}\label{one5}
\Phi(t) = \mu\big\{A \in U(N) : \vphi(A) < t\big\}
\end{equation}
for the corresponding distribution function.

\begin{lemma}\label{lemone1}  Let $A$ be a matrix in the group $U(N)$.  If
\begin{equation*}\label{one15}
\{\alpha_n: n = 1, 2, \dots , N\} \subseteq \T
\end{equation*}
are the eigenvalues of $A$, then
\begin{equation}\label{one17}
\vphi(A)^2 = 2 - 2\min\{\Re(\alpha_n) : n = 1, 2, \dots , N\}. 
\end{equation}
\end{lemma}

\begin{proof}  At each (column) vector $\bx$ in $X$, we have $|\bx|_2 = |A\bx|_2 = 1$.  It follows that
\begin{equation*}\label{one21}
|A\bx - \bx|_2^2 = \langle A\bx - \bx, A\bx - \bx \rangle = 2 - \bx^*(A^* + A)\bx,
\end{equation*}
where $\bx^*$ is the complex conjugate transpose of $\bx$, and similarly for $A^*$.  The matrix $A^* + A$ is self-adjoint.  Therefore,
by the spectral theorem, there exists an $N \times N$ unitary matrix $V$, and an $N \times N$ real diagonal matrix $D$, such that
\begin{equation*}\label{one23}
A^* + A = V^*D V,\quad\text{and}\quad D = [\omega_n],
\end{equation*} 
where 
\begin{equation*}\label{one25}
\{\omega_n : n = 1, 2, \dots , N\} = \{2\Re(\alpha_n) : n = 1, 2, \dots , N\}
\end{equation*}
are the eigenvalues of $A^* + A$.  Setting $\bwy = V\bx$ we find that $|\bwy|_2^2 = |\bx|_2^2 = 1$, and
\begin{align}\label{one27}
\begin{split}
|A\bx - \bx|_2^2 &= 2 - \bwy^*D\bwy\\
                             &= 2 - \sum_{n=1}^N \omega_n |y_n|^2\\ 
                             &\le 2 - 2\min\{\Re(\alpha_n) : n = 1, 2, \dots , N\}.
\end{split}                             
\end{align}
If
\begin{equation*}\label{one29}
\omega_m = \min\{\omega_n : n = 1, 2, \dots , N\},
\end{equation*}
then there is equality in the inequality (\ref{one27}) when $\bwy = \be_m$ is the $m$th standard basis vector.  As the vector
$V^*\be_m$ belongs to $X$, the identity (\ref{one17}) follows.
\end{proof}

If $z$ is a complex number we write $e(z) = e^{2\pi i z}$.

\begin{lemma}\label{lemone2}  Let $w$, $x$ and $y$ be real numbers such that $|w| \le \h$, $|x| \le \h$, and $|y| \le \h$.
Then we have
\begin{equation}\label{one33}
(2w)^2 \bigl|e(x) - e(y)\bigr|^2 \le \bigl|e(2wx) - e(2wy)\bigr|^2.
\end{equation}
\end{lemma}

\begin{proof}  If $u$ and $v$ are real numbers such that $0 \le |u| \le |v| \le 1$, then it follows from the convergent infinite product
\begin{equation}\label{one35}
\biggl(\frac{\sin \pi z}{\pi z}\biggr) = \prod_{n=1}^{\infty}\biggl(1 - \frac{z^2}{n^2}\biggr),
\end{equation}
that
\begin{equation*}\label{one37}
0 \le \biggl(\frac{\sin \pi v}{\pi v}\biggr) \le \biggl(\frac{\sin \pi u}{\pi u}\biggr).
\end{equation*}
By hypothesis,
\begin{equation*}
0 \le |2wx - 2wy| \le |x-y| \le 1.
\end{equation*}
Therefore we get
\begin{equation*}\label{one38}
0 \le \biggl(\frac{\sin \pi (x - y)}{\pi (x - y)}\biggr) \le \biggl(\frac{\sin \pi (2wx -2wy)}{\pi (2wx - 2wy)}\biggr),
\end{equation*}
and
\begin{equation}\label{one39}
(2w)^2 \bigl(\sin \pi (x-y)\bigr)^2 \le \bigl(\sin \pi (2wx-2wy)\bigr)^2.
\end{equation}
Then (\ref{one39}) is equivalent to (\ref{one33}).
\end{proof}

\begin{lemma}\label{lemone3}  If $0 < t \le 2$ then the distribution function $\Phi$ defined by {\rm (\ref{one5})}, satisfies the inequality
\begin{equation}\label{one45}
\biggl(\frac{t}{\pi}\biggr)^{N^2} \le \Phi(t).
\end{equation}
\end{lemma}

\begin{proof}  Let $\bx$ be a (column) vector in $\R^N$ with coordinates $x_1, x_2, \dots , x_N$, and let
$E : \R^N \rightarrow \C$ be defined by the Vandermonde determinant
\begin{equation*}
E(\bx) = \det\bigl(e\bigl((n-1) x_m\bigr)\bigr) = \prod_{1 \le m < n \le N} \bigl(e(x_n) - e(x_m)\bigr).
\end{equation*}  
As each function $x \mapsto e(x)$ is periodic with period $1$, the function $E$ is well defined on the compact quotient group
$(\R/\Z)^N$.  Let $\nu$ denote Haar measure on $(\R/\Z)^N$ normalized so that $\nu\bigl((\R/\Z)^N\bigr) = 1$.  Using the 
determinant representation for $E(\bx)$, we get the Fourier expansion
\begin{equation*}\label{one49}
E(\bx) = \sum_{\sigma \in S_N} \sgn(\sigma) \prod_{m = 1}^N e\bigl((\sigma(m) - 1) x_m\bigr),
\end{equation*}
where the sum is over the collection of all permutations 
\begin{equation*}
\sigma : \{1, 2, \dots , N\} \rightarrow \{1, 2, \dots , N\}
\end{equation*}
in the symmetric group $S_N$.  Then using Parseval's identity we find that
\begin{equation}\label{one53}
\frac{1}{N!} \int_{(\R/\Z)^N} |E(\bx)|^2\ \dnu(\bx) = 1.
\end{equation}

Let $F : U(N) \rightarrow \C$ be an integrable function with respect to the Haar measure $\mu$.  Assume that $F$ is constant on
conjugacy classes, so that
\begin{equation*}\label{one54}
F\bigl(B^{-1}AB\bigr) = F(A)
\end{equation*}
for all $A$ and $B$ in $U(N)$.  We recall that $F$ is constant on conjugacy classes
if and only if $F(A)$ depends only on the eigenvalues of $A$.  Let $A$ be a matrix in $U(N)$ with eigenvalues
\begin{equation*}\label{one57}
\{\alpha_n: n = 1, 2, \dots , N\} \subseteq \T.
\end{equation*}
Let $\bx = (x_n)$ be the unique point in $(\R/\Z)^N$ such that $\alpha_n = e(x_n)$.  If $F : U(N) \rightarrow \C$ is constant on 
conjugacy classes, then $F$ induces a function $\widetilde{F} : (\R/\Z)^N \rightarrow \C$, by $F(A) = \widetilde{F}(\bx)$.
Moveover, the function $\widetilde{F}$ is integrable with respect to Haar measure $\nu$,
and by the Weyl integration formula (see \cite[Chapter 5]{katz1999}) we have
\begin{equation}\label{one51}
\int_{U(N)} F(A)\ \dmu(A) = \frac{1}{N!} \int_{(\R/\Z)^N} \widetilde{F}(\bx) |E(\bx)|^2\ \dnu(\bx).
\end{equation}

Again let $A$ be a matrix in $U(N)$ with eigenvalues
\begin{equation*}
\{\alpha_n: n = 1, 2, \dots , N\} \subseteq \T,
\end{equation*}
and let $\bx = (x_n)$ be the point in $(\R/\Z)^N$ such that $\alpha_n = e(x_n)$.  By Lemma \ref{lemone1} the inequality $\vphi(A) < t$ 
holds if and only if 
\begin{equation*}\label{one69}
1 - \hh t^2 < \Re(\alpha_n) = \cos 2\pi x_n
\end{equation*}
for each $n = 1, 2, \dots , N$.  Therefore, if
\begin{equation*}\label{one73}
K(t) = \big\{\bx \in (\R/\Z)^N : 1 - \hh t^2 < \cos 2\pi x_n\big\},
\end{equation*}
then by Lemma \ref{lemone1} and the Weyl integration formula
(\ref{one51}), we have
\begin{equation*}\label{one77}
\Phi(t) = \mu\big\{A \in U(N) : \vphi(A) < t\big\} = \frac{1}{N!} \int_{K(t)} |E(\bx)|^2\ \dnu(\bx).
\end{equation*}
Alternatively, let $w$ be the unique real number such that $0 < w \le \h$,
\begin{equation*}\label{one81}
1 - \hh t^2 = \cos 2\pi w,\quad\text{and}\quad t = 2 \sin \pi w.
\end{equation*}
We will use the unit cube
\begin{equation*}\label{one85}
C_N = \{\bwy \in \R^N : |y_n| \le \hh\}
\end{equation*}
as a fundamental domain for the quotient group $(\R/\Z)^N$.  Then we have
\begin{equation*}\label{one86}
K(t) \cap C_N = \big\{\bwy \in \R^N : |y_n| < w\big\},
\end{equation*}
and
\begin{equation}\label{one87}
\Phi(t) = \frac{1}{N!} \int_{\{\bwy : |y_n| < w\}} |E(\bwy)|^2\ \dy.
\end{equation}
The inequality (\ref{one33}), implies that
\begin{align}\label{one89}
\begin{split}
(2w)^{N^2 - N} |E(\bwy)|^2 &= \prod_{1 \le m < n \le N} (2w)^2 \bigl|e(y_n) - e(y_m)\bigr|^2 \\
                                            &\le \prod_{1 \le m < n \le N} \bigl|e(2wy_n) - e(2wy_m)\bigr|^2 \\
                                            &= E(2w\bwy)
\end{split}
\end{align}
at each point $\bwy$ in $C_N$.  It follows that
\begin{align}\label{one93}
\begin{split}
(2w)^{N^2} &= \frac{(2w)^N}{N!} \int_{C_N} (2w)^{N^2 - N} |E(\bwy)|^2\ \dy\\
                     &\le \frac{(2w)^N}{N!} \int_{C_N} |E(2w\bwy)|^2\ \dy\\
                     &= \frac{1}{N!} \int_{\{\bx : |x_n| < w\}} |E(\bx)|^2\ \dx\\
                     &= \Phi(t).
\end{split}
\end{align}
To complete the proof we use the elementary inequality
\begin{equation}\label{one97}
t = 2 \sin \pi w \le 2\pi w.
\end{equation}
The inequality (\ref{one45}) plainly follows from (\ref{one93}) and (\ref{one97}).
\end{proof}

Let $\A \subseteq U(N)$ be a finite, nonempty subset of cardinality $|\A| \ge 2$.  Combining the inequalities 
(\ref{group265}) and (\ref{one45}), we conclude that
\begin{equation*}\label{one105}
\biggl(\frac{\delta(\A)}{2 \pi} \biggr)^{N^2} \le \Phi\bigl(\hh\delta(\A)\bigr) \le |\A|^{-1},
\end{equation*}
and this verifies Theorem \ref{thmintro1}.

Similarly, let $A$ and $B$ be elements of $U(N)$, and let  $\delta_{J,K}(A, B)$ be defined by (\ref{intro43}).  Then 
(\ref{group331}) implies that
\begin{equation*}\label{one111}
\biggl(\frac{\delta_{J, K}(A, B)}{2 \pi}\biggr)^{N^2} \le (J + 1)^{-1} (K + 1)^{-1},
\end{equation*}
and this proves the inequality (\ref{intro49}) in the statement of Theorem \ref{thmintro2}.


\end{document}